\documentclass[12pt,reqno]{amsart}
\usepackage{amsmath, amsthm, amssymb, cite, bm, eufrak,enumitem}

\advance \topmargin by -\headheight
\advance \topmargin by -\headsep
     
\evensidemargin \oddsidemargin
\newtheorem{theorem}{Theorem}

\newtheorem{corollary}[theorem]{Corollary}

\theoremstyle{definition}

\title{Almost Equal Summands in Waring's Problem with Shifts}
\author{Kirsti D. Biggs}
\address{School of Mathematics, University of Bristol, University Walk, Clifton, Bristol, BS8 1TW, United Kingdom}
\email{kirsti.biggs@bristol.ac.uk}
\thanks{The author is supported by EPSRC Doctoral Training Partnership EP/M507994/1}

\begin{document}

\begin{abstract}
A result of Wright from 1937 shows that there are arbitrarily large natural numbers which cannot be represented as sums of $s$ $k$th powers of natural numbers which are constrained to lie within a narrow region. We show that the analogue of this result holds in the shifted version of Waring's problem.
\end{abstract}

\maketitle

Waring's problem with shifts asks whether, given $k,s\in\mathbb{N}$ and $\eta\in(0,1]$, along with shifts $\theta_1,\dotsc,\theta_s\in(0,1)$ with $\theta_1\not\in\mathbb{Q}$, we can find solutions in natural numbers $x_i$ to the following inequality, for all sufficiently large $\tau\in\mathbb{R}$:
\begin{equation}\label{ineq}
\left|(x_1-\theta_1)^k+\dotsc+(x_s-\theta_s)^k-\tau\right|<\eta.
\end{equation}
This problem was originally studied by Chow in \cite{chowWP}. In \cite{Biggs}, the author showed that an asymptotic formula for the number of solutions to (\ref{ineq}) can be obtained whenever $k\geq 4$ and $s\geq k^2+(3k-1)/4$. The corresponding result for $k=3$ and $s\geq 11$ is due to Chow in \cite{cubes}.

An interesting variant is to consider solutions of (\ref{ineq}) subject to the additional condition
\begin{equation*}
\left|x_i-(\tau/s)^{1/k}\right|<y(\tau),\quad(1\leq i\leq s),
\end{equation*}
for some function $y(\tau)$. In other words, we are confining our variables to be within a small distance of the ``average'' value.

In 1937, Wright studied this question in the setting of the classical version of Waring's problem, and proved in \cite{equalsq} that there exist arbitrarily large natural numbers $n$ which cannot be represented as sums of $s$ $k$th powers of natural numbers $x_i$ satisfying the condition $\left|x_i^k-n/s\right|<n^{1-1/2k}\phi(n)$ for $1\leq i\leq s$, no matter how large $s$ is taken. Here, $\phi(n)$ is a function satisfying $\phi(n)\to 0$ as $n\to\infty$.

In \cite{Daemen2} and \cite{Daemen1}, Daemen showed that if we widen the permitted region slightly, we can once again guarantee solutions in the classical case. Specifically, he obtains a lower bound on the number of solutions under the condition
\begin{equation*}
\left|x_i-(n/s)^{1/k}\right|<cn^{1/2k},\quad(1\leq i\leq s),
\end{equation*}
for a suitably large constant $c$, and an asymptotic formula under the condition
\begin{equation*}
\left|x_i-(n/s)^{1/k}\right|<n^{1/2k+\epsilon},\quad(1\leq i\leq s).
\end{equation*}

In this note, we show that (a slight strengthening of) Wright's result remains true in the shifted case. Specifically, we prove the following.
\begin{theorem}\label{Thm}
Let $s,k\geq 2$ be natural numbers.
Fix $\bm{\theta}=(\theta_1,\dotsc,\theta_s)\in(0,1)^s$, and let $c,c'>0$ be suitably small constants which may depend on $s,k$ and $\bm{\theta}$. There exist arbitrarily large values of $\tau\in\mathbb{R}$ which cannot be approximated in the form (\ref{ineq}), with $0<\eta<c\tau^{1-2/k}$, subject to the additional condition that $\left|x_i-(\tau/s)^{1/k}\right|<c'\tau^{1/2k}$ for $1\leq i\leq s$.
\end{theorem}

\begin{proof}
This follows the structure of Wright's proof in \cite{equalsq}, with minor adjustments to take into account the shifts present in our problem. As such, for $m\in\mathbb{N}$, we let $\tau_m=sm^{k}+km^{k-1}(s-\sum_{i=1}^s \theta_i)$, and we note that $\tau_m\to\infty$ as $m\to\infty$. Throughout the proof, we allow $c_1,c_2,\dotsc$ to denote positive constants which do not depend on $m$, although they may depend on the fixed values of $s,k, \bm{\theta}, c$ and $c'$. We also note that $\eta<c\tau^{1-2/k}$ implies that $\eta\ll m^{k-2}$.

Suppose $\tau_m$ satisfies (\ref{ineq}) with $0<\eta<c\tau_m^{1-2/k}$ and $\left|x_i-(\tau_m/s)^{1/k}\right|<c'\tau_m^{1/2k}$ for $1\leq i\leq s$. We write $x_i=m+a_i$, and observe that
\begin{align*}
m^{k-1}\left|a_i\right|&=m^{k-1}\left|x_i-m\right|\\
&\leq m^{k-1}\Big(\left|x_i-(\tau_m/s)^{1/k}\right|+\left|(\tau_m/s)^{1/k}-m\right|\Big)\\
&\leq c'm^{k-1}\tau_m^{1/2k}+\left|\tau_m/s-m^k\right|.
\end{align*}
Using the definition of $\tau_m$, we obtain
\begin{align*}
m^{k-1}\left|a_i\right|&\leq c_1m^{k-1}m^{1/2}+km^{k-1}(1-s^{-1}\sum_{i=1}^s \theta_i),
\end{align*}
and therefore $\left|a_i\right|\leq c_2 m^{1/2}$ for $1\leq i\leq s$. Expanding (\ref{ineq}), we see that
\begin{align}\label{binexpn}
\eta&> \left|\sum_{i=1}^s (x_i-\theta_i)^k - \tau_m\right|\nonumber\\
&=\left| \sum_{i=1}^s (m+a_i-\theta_i)^k - \Big(sm^{k}+km^{k-1}(s-\sum_{i=1}^s \theta_i)\Big)\right|\\
&\geq km^{k-1}\left|s-\sum_{i=1}^s a_i \right|-\left|\sum_{j=2}^k\binom{k}{j}m^{k-j}\sum_{i=1}^s(a_i-\theta_i)^j\right|.\nonumber
\end{align}
Rearranging, this gives
\begin{align*}
\left|s-\sum_{i=1}^s a_i \right|&<\eta k^{-1}m^{1-k}+\left|\sum_{j=2}^k\binom{k}{j}k^{-1}m^{1-j}\sum_{i=1}^s(a_i-\theta_i)^j\right|\\
&\leq  \eta k^{-1}m^{1-k}+\sum_{j=2}^k\binom{k}{j}k^{-1}m^{1-j} s(c_3 m^{1/2})^j\\
&\leq c_4.
\end{align*}
By choosing our original $c,c'$ to be sufficiently small, we may conclude  that $c_4\leq 1$, which implies that $\sum_{i=1}^s a_i = s$.
Substituting this back into (\ref{binexpn}), when $k=2$ we obtain
\begin{align*}
\eta&>\binom{k}{2}m^{k-2}\sum_{i=1}^s(a_i-\theta_i)^2,
\end{align*}
and consequently
\begin{align*}
\sum_{i=1}^s(a_i-\theta_i)^2 < c_5,
\end{align*}
which is a contradiction if we choose $c,c'$ sufficiently small, since we have $\sum_{i=1}^s(a_i-\theta_i)^2\gg 1$.

When $k\geq 3$, we obtain
\begin{align*}
\eta&>\left|\sum_{j=2}^k\binom{k}{j}m^{k-j}\sum_{i=1}^s(a_i-\theta_i)^j\right|\\
&\geq \binom{k}{2}m^{k-2}\sum_{i=1}^s(a_i-\theta_i)^2-\left|\sum_{j=3}^k\binom{k}{j}m^{k-j}\sum_{i=1}^s(a_i-\theta_i)^j\right|.
\end{align*}
Consequently,
\begin{align*}
\binom{k}{2}m^{k-2}&\sum_{i=1}^s(a_i-\theta_i)^2<\eta+\sum_{j=3}^k\binom{k}{j}m^{k-j}\sum_{i=1}^s\left|a_i-\theta_i\right|^j\\
&\leq \eta+\sum_{j=3}^k\binom{k}{j}m^{k-j}(c_3 m^{1/2})^{j-2}\sum_{i=1}^s (a_i-\theta_i)^2\\
&\leq \eta+c_6 m^{k-5/2}\sum_{i=1}^s (a_i-\theta_i)^2,
\end{align*}
and so
\begin{align*}
\sum_{i=1}^s(a_i-\theta_i)^2&<c_7+c_8 m^{-1/2}\sum_{i=1}^s(a_i-\theta_i)^2,
\end{align*}
which is again a contradiction when $m$ is large.

We conclude that for all sufficiently large $m$, it is impossible to approximate $\tau_m$ in the manner claimed. This completes the proof.
\end{proof} 

\begin{corollary}
For $s,k\geq 2$ natural numbers, $\bm{\theta}=(\theta_1,\dotsc,\theta_s)\in(0,1)^s$, and suitably small constants $C,C'>0$, there exist arbitrarily wide gaps between real numbers $\tau$ for which the system 
\begin{gather}\label{newsys}
\begin{gathered}
\left|(x_1-\theta_1)^k+\dotsc+(x_s-\theta_s)^k-\tau\right|<C\tau^{1-2/k}\\
\left|x_i-(\tau/s)^{1/k}\right|<C'\tau^{1/2k},\quad (1\leq i\leq s)
\end{gathered}
\end{gather}
has a solution in natural numbers $x_1,\dotsc,x_s$.
\end{corollary}
\begin{proof}
By Theorem \ref{Thm}, we fix $\tau_0\in\mathbb{R}$ such that there is no solution in natural numbers $x_1,\dotsc,x_s$ to $\left|(x_1-\theta_1)^k+\dotsc+(x_s-\theta_s)^k-\tau_0\right|<c\tau_0^{1-2/k}$ with $\left|x_i-(\tau_0/s)^{1/k}\right|<c'\tau_0^{1/2k}$ for $1\leq i\leq s$.

 Let $0<\delta\leq C_0\tau_0^{1-2/k}$ for some $C_0>0$, and let $\tau\in[\tau_0-\delta,\tau_0+\delta]$. Let $C,C'>0$ be suitably small constants depending on $c,c'$ and $C_0$ to be chosen later, and suppose that $x_1\dotsc,x_s\in\mathbb{N}$ are such that (\ref{newsys}) is satisfied.
 
We have 
\begin{align*}
\left|(\tau/s)^{1/k}-(\tau_0/s)^{1/k}\right|&\leq s^{-1/k}\left|(\tau_0-\delta)^{1/k}-\tau_0^{1/k}\right|\\
&\leq C_1 \delta\tau_0^{1/k-1},
\end{align*}
and consequently
\begin{align*}
\left|x_i-(\tau_0/s)^{1/k}\right|&\leq \left|x_i-(\tau/s)^{1/k}\right|+\left|(\tau/s)^{1/k}-(\tau_0/s)^{1/k}\right|\\
&< C'\tau^{1/2k}+C_1 \delta\tau_0^{1/k-1}\\
&\leq C'(\tau_0+\delta)^{1/2k}+ C_1C_0\tau_0^{-1/k}\\
&\leq C_2 \tau_0^{1/2k}.
\end{align*}
We also see that
\begin{align*}
\left|\sum_{i=1}^s (x_i-\theta_i)^k-\tau_0\right|&\leq \left|\sum_{i=1}^s (x_i-\theta_i)^k-\tau\right|+\left|\tau-\tau_0\right|\\
&< C\tau^{1-2/k}+\delta\\
&\leq C(\tau_0+\delta)^{1-2/k}+C_0\tau_0^{1-2/k}\\
&\leq C_3 \tau_0^{1-2/k}.
\end{align*}
Choosing $C_0,C,C'$ small enough to ensure that $C_2\leq c'$ and $C_3\leq c$ gives a contradiction to our original choice of $\tau_0$. Consequently, there is no solution to (\ref{newsys}) in an interval of radius $\asymp \tau_0^{1-2/k}$ around $\tau_0$.
\end{proof}
The author would like to thank Trevor Wooley for his supervision, and the anonymous referee for useful comments.

\newcommand{\noop}[1]{}

\end{document}